%% file: Approx_bounds_poly.tex
\documentclass[]{scrartcl}
\usepackage[utf8]{inputenc}
\usepackage[letterpaper,bindingoffset=0cm,inner=2.5cm,outer=2.5cm,top=2.5cm,bottom=2.5cm]{geometry}
\usepackage{helvet}
\usepackage[T1]{fontenc}
\usepackage{graphicx}
\usepackage{amsmath,amsthm,amstext,amssymb,bm}
\usepackage{mathtools}
\usepackage{xcolor,color}
\usepackage{enumerate}
\usepackage{pgf}
\usepackage[toc,page]{appendix}
\usepackage[colorlinks=true]{hyperref}
\definecolor{mylinkcolor}{RGB}{0,0,130}
\hypersetup{colorlinks,allcolors=mylinkcolor,citecolor=mylinkcolor}
\usepackage{url}
\usepackage[backend=biber,%
                      bibencoding=utf8,%
                      citestyle=numeric,%
                      giveninits=true,%
                      isbn=false,%
                      sortcites=true,%
                      natbib=true,
					  maxnames=50,]{biblatex}
\usepackage[capitalize,nameinlink]{cleveref}
\crefformat{equation}{(#2#1#3)}
\usepackage{booktabs}
\usepackage{tabularx}
\usepackage{enumerate}
\usepackage[shortlabels]{enumitem}
\usepackage{multirow}
\usepackage{subfigure}
\usepackage[font=small,labelfont=bf]{caption}

\providecommand{\keywords}[1]{\textbf{\textit{Keywords ---}} #1}
\providecommand{\msc}[1]{\textbf{\textit{MSC 2010 ---}} #1}
\providecommand{\acknowledgements}[1]{\textbf{\textit{Acknowledgements ---}} #1}

\usepackage{color,colortbl}

\definecolor{darkred}{rgb}{0.3,0,0}
\definecolor{darkgreen}{rgb}{0,0.3,0}
\definecolor{darkblue}{rgb}{0,0,0.3}
\definecolor{pink}{rgb}{0.78,0.09,0.51} 

\definecolor{orange}{rgb}{1,0.6,0.0}
\definecolor{grey}{rgb}{0.4,0.4,0.4}
\definecolor{lightgray}{gray}{0.9}
\definecolor{aquamarine}{rgb}{0.4,0.8,0.65}

\newcommand{\R}{\mathbb{R}}
\newcommand{\field}{\mathbb{K}}
\newcommand{\N}{\mathbb{N}}

\newcommand{\bfx}{\bm{x}}
\newcommand{\xr}{\check{x}}
\newcommand{\bfxr}{\check{\bm{x}}}
\newcommand{\bfA}{\bm{A}}
\newcommand{\bfy}{\bm{y}}

\newcommand{\bfmu}{\bm{\mu}}

\newcommand{\nmu}{n_{\bfmu}}

\newcommand{\norm}[1]{||#1||_{\linSpace}}

\newcommand{\kkron}[2][k]{#2^{\otimes #1}}
\newcommand{\numTotal}[2]{t(#1,#2)}
\newcommand{\skkron}[2][k]{#2^{\otimes^\mathrm{s} #1}}
\newcommand{\kTensor}[1][k]{\bfA_#1}

\newcommand{\cU}{{U}}
\newcommand{\kolWidth}[3][n]{d_{#1}\left( #2;#3 \right)}

\newcommand{\linSpace}{V}
\newcommand{\linSpaceAlt}{U}
\newcommand{\subSet}{S}

\newcommand{\dist}{\mathrm{dist}}
\newcommand{\solmnf}{\mathcal{M}}
\renewcommand{\P}{\mathcal{P}}

\newcommand{\dec}[2][n]{\Gamma_{#1,#2}}

\DeclareMathOperator*{\img}{img}
\newcommand{\submnf}[2][n]{\widetilde{\mathcal{M}}_{#1,#2}}

\newcommand{\subspace}[2]{\mathcal{A}_{#1,#2}}
\newcommand{\vectorSet}[2]{\mathcal{V}_{#1,#2}}

\newcommand{\mapKron}[2]{K_{#1,#2}}
\newcommand{\mapLin}[2]{A_{#1,#2}}
\newcommand{\polyKolWidthOne}[4][n]{d^{\otimes}_{#1,#2}\left( #3;#4 \right)}

\newcommand{\decGen}[1][n]{d_{#1}}
\newcommand{\mapLinGen}[1]{A_{#1}}
\newcommand{\mapKronGen}[1]{K_{#1}}

\newcommand{\polyManiWidth}[4][n]{\delta^{\otimes}_{#1,#2} \left( #3;#4 \right)}
\newcommand{\xred}{\bfxr}

\newtheorem{lemma}{Lemma}
\newtheorem{example}{Example}
\newtheorem{definition}{Definition}
\newtheorem{remark}{Remark}
\newtheorem{corollary}{Corollary}
\newtheorem{theorem}{Theorem}

\title{Approximation Bounds for Model Reduction on Polynomially Mapped Manifolds}

\author{%
Patrick Buchfink\thanks{%
Institute of Applied Analysis and Numerical Simulation, University of Stuttgart, Pfaffenwaldring 57, 70569 Stuttgart, Germany.
(\url{patrick.buchfink,haasdonk@mathematik.uni-stuttgart.de})}
\and
Silke Glas\thanks{%
University of Twente, Department of Applied Mathematics, P.O. Box 217, 7500 AE Enschede, The Netherlands.
(\url{s.m.glas@utwente.nl})}
\and
Bernard Haasdonk\footnotemark[1]}

\begin{document}

\maketitle

\begin{abstract}
	\textbf{Abstract} For projection-based linear-subspace model order reduction (MOR), it is well known that the Kolmogorov $n$-width describes the best-possible error for a reduced order model (ROM) of size $n$. In this paper, we provide approximation bounds for ROMs on polynomially mapped manifolds. In particular, we show that the approximation bounds depend on the polynomial degree $p$ of the mapping function as well as on the linear Kolmogorov $n$-width for the underlying problem. This results in a Kolmogorov $(n,p)$-width, which describes a lower bound for the best-possible error for a ROM on polynomially mapped manifolds of polynomial degree $p$ and reduced size $n$.
\end{abstract}

\keywords{model order reduction, approximation bounds, polynomial mappings, Kolmogorov $n$-widths}

\msc{41A45, 41A46, 41A65, 65L70, 65M15, 65N15}

\newcommand{\sol}{x}
\newcommand{\solMOR}{\widetilde{x}}

\section{Introduction}

Model Order Reduction (MOR) is used to derive surrogate models for high-dimensional \emph{full-order models} (FOMs).
This allows to execute and speed up tasks which require to evaluate the FOM multiple times for different parameters (e.g.\ in parameter studies, sampling-based uncertainty quantification) or in real time (e.g.\ in model-based control).
We denote the (possibly parametric) FOM with $P(\bfmu)$ in dependence of an arbitrary but fixed \emph{parameter vector} $\bfmu \in \P$ from a given \emph{parameter domain} $\P \subset \R^{\nmu}, \nmu \in \N$.
Typically, the FOM is a parametric system of partial or ordinary differential equations and it is formulated on an $N$-dimensional Banach space $(\linSpace, \norm{\cdot})$ over a field $\field$, ($\field = \R$ or $\field = \mathbb{C}$), with large dimension $N \in \N \cup \{ \infty \}$.
The goal of MOR is to approximate the so-called solution manifold
\begin{align} \label{eq:parametric_solution_manifold}
	\solmnf := \left\{ \sol(\bfmu) \in \linSpace : x(\bfmu) \textnormal{ is solution of } P(\bfmu) \textnormal{ for parameter vector } \bfmu \in \P \right\} \subset \linSpace.
\end{align}
To this end, classical MOR determines a low-dimensional subspace $\linSpace_n \subset \linSpace$ with $\dim(\linSpace_n) = n \ll N$ and an efficiently computable \emph{reduced-order model (ROM)} $P_n(\bfmu)$.
The ROM is used to compute a reduced solution to approximate the FOM solution $\sol(\bfmu) \in \solmnf$.
The quality of this approximation can be bounded from below by the \emph{Kolmogorov $n$-widths}.
These quantify how well a given subset $\subSet \subset \linSpace$ can be approximated by an $n$-dimensional linear subspace of $\linSpace$ with
\begin{align*}
	d_n(\subSet;\linSpace) := \inf_{\substack{\linSpaceAlt_n \subset \linSpace,\\ \textnormal{dim}(\linSpaceAlt_n) = n}} \sup_{x \in \subSet} \inf_{x_n \in \linSpaceAlt_n} \| x-x_n\|_\linSpace.
\end{align*}
Since classical MOR relies on the approximation in a linear subspace of dimension $n$, it is well-known that the best-possible approximation of all elements in the solution manifold $\solmnf$ is bounded from below by $d_n(\solmnf;\linSpace)$.

For some problem classes, there have been analytical results for the behavior of the Kolmogorov $n$-widths for increasing $n$. E.g.\ for linear coercive elliptic PDEs
(a) with one parameter, it has been shown that $d_n(\solmnf;\linSpace) $ decays at least exponentially, i.e., $d_n(\solmnf;\linSpace)  \le C\exp^{-\gamma n}$ for $C,\gamma >0$, see \cite{Maday2002,Maday2002a}, or (b) with $d \in \N$ parameters,
\cite{ohlberger2016reduced,bachmayr2017kolmogorov} prove for affinely decomposable problems that the decay is at least $d_n(\solmnf; \linSpace) \le C \exp^{-c n^{\gamma}}$ for $C,c,\gamma > 0$.
For linear transport equations or linear wave equations, it is known that $d_n(\solmnf;\linSpace)$ can exhibit slow decays with a rate of at most $n^{-1/2}$, see \cite{greif2019decay,ohlberger2016reduced}. Especially for the latter case, there have been various attempts to ``break'' the Kolmogorov $n$-width barrier by, e.g., considering model reduction on manifolds.

The object of interest in the Kolmogorov $n$-width are the linear subspaces $\linSpace_n \subset \linSpace$.
These can be characterized by a basis $\{ v_i \}_{i=1}^n \subset \linSpace_n$ via linear combination of basis vectors $v_i \in \linSpace_n$ with basis coefficients $\xr_i \in \field$.
In MOR, the basis coefficients $\bfxr := (\xr_i)_{i=1}^n \in \field^n$ are referred to as \emph{reduced coordinates} and are determined by solving the according ROM.
In this paper, we are looking at Kolmogorov $n$-widths for a special type of submanifold for which the basis coefficients $\xr_i$ are obtained from a polynomial of degree $p \in \N_0$.
We refer to these submanifolds as \emph{polynomially mapped}.
Moreover, we introduce an analogue to the Kolmogorov $n$-width,
which we refer to as \emph{polynomial Kolmogorov $(n,p)$-width} which additionally depends on the order $p$ of the polynomial. For an overview of different versions of nonlinear widths, we refer to \cite{Cohen2023}. Additionally, we relate our polynomial Kolmogorov $(n,p)$-width to a polynomial analogue of the nonlinear manifold width, first introduced in \cite{devore1989optimal}, at the end of this paper. 

Note, that in previous works, polynomial approximations are ubiquitous in the field of numerical approximation of high-dimensional problems. For example in the work \cite{cohen2011analytic}, parametric PDEs are approximated with a polynomial function. However, this approach is polynomial in the parameter vector $\bfmu$ while our approximation is polynomial in the reduced coordinates which may depend arbitrarily complex on the parameter vector. It has been shown in \cite{bachmayr2017kolmogorov} for classical MOR (i.e. $p$=1 in our case) that such methods can perform significantly better than the methods from \cite{cohen2011analytic}.

Moreover, there have been recent approaches using polynomial approximation for the reduced coordinates. 
In the work of \cite{gu2011model}, the general idea of using nonlinear mapping functions and especially quadratic mapping functions ($p=2$) has been described, but the model reduction is only performed with piece-wise linear manifolds.
Recently, the idea of model reduction with quadratically embedded manifolds has been introduced in \cite{jain2017quadratic, rutzmoser2017generalization} for structural nonlinear dynamics.
In these references, the approaches are introduced for a special class of second-order dynamical systems by e.g., using a linearized problem to compute vibration modes as linear part of the basis and then constructing the quadratic extension via modal derivatives. A different possibility to derive the modal derivatives in the context of model reduction on quadratically embedded manifolds has been derived in \cite{cruz2020}. In \cite{BARNETT2022111348} the authors perform projection based model reduction on quadratically embedded manifolds, but their approach holds for more general settings, e.g., first-order equations and transport-dominated flow problems. The same idea of using a quadratic mapping function has been used in \cite{GEELEN2023115717}, but there the authors perform operator inference, i.e., the ROM is learned from data instead of projecting the FOM. In \cite{ISSAN2023111689}, a shifted operator inference using quadratic mapping functions has been used to predict solar winds. In \cite{BennerHeilandDuffPawan2022}, also operator interference is used, but the approach utilizes a state-dependent mass matrix that depends on the derivative of the mapping function.
A structure-preserving technique for Hamiltonian systems has been introduced in \cite{Sharma2023}.

In the present work, we are particularly interested in what is the best-possible approximation error of a solution manifold by a polynomially mapped manifold. We will derive an upper as well as a lower bound for this approximation error, with both bounds being related to the classical Kolmogorov $n$-width of the problem at hand.

\section{Approximation Bounds on Polynomial Manifolds}

In this section, we start by defining polynomial mappings and polynomially mapped manifolds. Then, we show that these polynomially mapped manifolds are contained in a linear subspace, which leads to a lower and an upper bound for the polynomial Kolmogorov $(n,p)$-widths of these manifolds. We close this section by stating how this impacts certain decay rates for which a classical Kolmogorov $n$-width is known. 

\subsection{Polynomial Mappings}
Consider a vector space $\linSpace$ over a field $\field$, ($\field = \R$ or $\field = \mathbb{C}$).
The \emph{Kronecker product (of order $k \geq 0$)} of a vector $\bfxr = (\xr_i)_{i=1}^n \in \field^n$ is denoted with
\begin{align*}
\kkron{\bfxr} :=&
\Big(
\underbrace{
	\xr_1
	\cdots
	\xr_1
}_{k \text{ terms}},\;
\underbrace{
	\xr_1
	\cdots
	\xr_1
	\cdot
	\xr_2
}_{k \text{ terms}},\;
\dots,\;
\underbrace{
	\xr_n
	\cdots
	\xr_n
}_{k \text{ terms}}
\Big)
\in \field^{n^k}
\qquad \text{for $k\geq1$},&
	\kkron[0]{\bfxr} := 1 \in \field.
\end{align*}
Due to the commutativity of the multiplication of elements in $\field$,
the Kronecker product contains redundant entries.
Thus, we consider the \emph{symmetric Kronecker product} $\skkron{\bfxr} \in \field^{m(n,k)}$ which neglects duplicate terms in $\kkron{\bfxr}$ and results in
$$m(n,k):=\binom{n+k-1}{k} \leq n^k$$
entries.

For $n \in \N$, $p \in \N_0$ and a given set of vectors
$\vectorSet{n}{p} := \left\{ v_{kj} \in \linSpace \,\big|\, 0 \leq k \leq p, \, 1 \leq j \leq m(n, k) \right\} \subset \linSpace$,
we consider \emph{polynomial mappings with degree $p\ge 0$}
\begin{align*}
	\dec[n]{p}: \field^n \to \linSpace,\quad
	\bfxr \mapsto
\sum_{k=0}^p \sum_{j=1}^{m(n, k)} \left( \skkron{\bfxr} \right)_j v_{kj}
\end{align*}
which sums over all symmetric Kronecker products from order $0$ to $p$.
Following the notation in MOR,
we call $\bfxr \in \field^n$ the \emph{reduced coordinates}.
In total,
\begin{align*}
	\numTotal{n}{p} := \sum_{k=0}^{p} m(n,k)
\end{align*}
vectors are used in the polynomial mapping.
The image of a polynomial mapping of order $p$ defines an at most $n$-dimensional submanifold of $\linSpace$
which we call \emph{polynomially mapped submanifold} and denote in the following as 
\begin{align*}
	\submnf[n]{p} := \img(\dec[n]{p}) \subset \linSpace.
\end{align*}

We can then immediately show the following lemma, which will be needed later to derive the approximation bounds. 
\begin{lemma}[Intermediate Linear Mapping]\label{lem:lin_embed}
	The image of a polynomial mapping of order $p$ is contained in a $\numTotal{n}{p}$-dimensional subspace of $\linSpace$.
\end{lemma}

\begin{proof}
	Firstly, we rewrite the polynomial mapping as the composition $\dec[n]{p} := \mapLin{n}{p} \circ \mapKron{n}{p}$ of a nonlinear map
	\begin{align*}
		\mapKron{n}{p}: \field^n \to \field^{\numTotal{n}{p}},
\quad
\bfxr \mapsto (\skkron{\bfxr}, \skkron[k-1]{\bfxr}, \dots, \skkron[2]{\bfxr}, \bfxr, 1)
	\end{align*}
	which generates all symmetric Kronecker products and a linear map
	\begin{align*}
		\mapLin{n}{p}: \field^{\numTotal{n}{p}} \to \linSpace,\;
		\bfy \mapsto \sum_{v_i \in \vectorSet{n}{p}} (\bfy)_i v_i.
	\end{align*}
	But then, the submanifold $\submnf[n]{p} = \img(\dec[n]{p}) \subset \img(\mapLin{n}{p}) =: \subspace{n}{p}$ is a subset of the $\numTotal{n}{p}$-dimensional subspace $\subspace{n}{p}$.
\end{proof}

\begin{example}[Finite-Dimensional Vector Spaces]
	In the case of $\linSpace = \field^N$ for some $N \in \N$,
	the vectors $v_{kj} \in \vectorSet{n}{p}$ for the polynomial mapping may simply be given by the columns of $(p+1)$ \emph{mapping matrices} $\kTensor := [v_{k\, 1}, \dots, v_{k\, m(n,k)}] \in \field^{N \times m(n,k)}$, $0 \leq k \leq p$, such that
	\begin{align*}
		\dec[n]{p}(\bfxr) = \sum_{k=0}^{p}  \kTensor \skkron{\bfxr},
	\end{align*}
	In this case, the linear map from \cref{lem:lin_embed} stacks all $0\leq k \leq p$ mapping matrices $\kTensor$ in its columns
	\begin{align*}
		&\mapLin{n}{p}(\bfy) = \bfA \bfy,&
&\bfA := \left[ \kTensor[p],\, \dots,\, \kTensor[0]\right] \in \field^{N \times \numTotal{n}{p}}&
&\text{such that}&
&\dec[n]{p}(\bfxr) = \bfA \mapKron{n}{p}(\bfxr).
	\end{align*}
	Thus, the definition of a polynomial mapping covers
	\begin{enumerate}
		\item classical MOR on affine or linear subspaces with
		$\dec[n]{1}(\bfxr) = \kTensor[1] \bfxr + \kTensor[0],$
		\item MOR on quadratic manifolds, as e.g.\ in \cite{ISSAN2023111689,BARNETT2022111348,GEELEN2023115717,BennerHeilandDuffPawan2022}, with
		$\dec[n]{2}(\bfxr) = \kTensor[2] \skkron[2]{\bfxr} + \kTensor[1] \bfxr + \kTensor[0].$
	\end{enumerate}
\end{example}

\subsection{Bounds for the Approximation Error using Polynomial Mapping}

We start by recalling the classical Kolmogorov $n$-width
and transfer its concept to polynomially mapped manifolds.
Afterwards, we show that the polynomial Kolmogorov $(n,p)$-width
can be bounded from below and above by classical Kolmogorov widths.

\begin{definition}[Worst Best-approximation Error]
	Let $(\linSpace, \norm{\cdot})$ be a normed vector space.
	For two sets $\subSet,T \subseteq \linSpace$,
	we call
	\begin{align*}
        \dist(\subSet, T) := \sup_{s \in \subSet}\ \inf_{t \in T} \norm{s - t}
	\end{align*}
	the \emph{worst best-approximation error of $\subSet$ in $T$}.
\end{definition}
Next, we are interested in how well a subset $\subSet\subseteq \linSpace$ can be approximated by an $n$-dimensional linear subspace of $\linSpace$. This measure is known as the Kolmogorov $n$-width and the idea was first formulated in \cite{Kolmogorov1936}, although we refer here to a later definition. 
\begin{definition}[(Classical) Kolmogorov $n$-width {\cite[Chapter II, Definition 1.1]{Pinkus1985}}]
    Let $(\linSpace, \norm{\cdot})$ be a normed vector space
    and let $\subSet \subset \linSpace$ be a subset.
    Then, the \emph{Kolmogorov $n$-width}
    \begin{align}\label{Eq:kolWidth}
        \kolWidth[n]{\subSet}{\linSpace}
:= \inf_{\substack{
    \cU \subseteq \linSpace \text{ subspace}\\
    \dim(\cU) \leq n
}}\,
        \dist(\subSet, \cU),
    \end{align}
    measures the theoretically optimal worst best-approximation error of $\subSet$ achievable by some at most $n$-dimensional subspace $\cU$ of $\linSpace$. If $ \kolWidth[n]{\subSet}{\linSpace} =    \dist(\subSet, \cU),$ for some subspace $\cU$ of dimension at most $n$, then $\cU$ is said to be an \emph{optimal subspace} for $ \kolWidth[n]{\subSet}{\linSpace}$. 
\end{definition}

The Kolmogorov $n$-widths are an established measure to argue how well the solution manifold $\solmnf \subset \linSpace$ from \eqref{eq:parametric_solution_manifold} can be approximated by classical MOR. Since the approximation of the ROM in classical MOR is determined in an $n$-dimensional subspace $\linSpace_n \subset \linSpace$, the best-possible approximation is limited from below by the Kolmogorov $n$-width by construction, i.e.\
\begin{align}\label{eq:relation_kolmogorov_mor}
	\dist(\solmnf, \linSpace_n)
\geq \kolWidth[n]{\solmnf}{\linSpace}.
\end{align}

In the following, we generalize the classical Kolmogorov $n$-width to the \emph{polynomial Kolmogorov $(n,p)$-width}, i.e., we are interested in the best-possible approximation error for a ROM constructed by a polynomial mapping. 
\begin{definition}[Polynomial Kolmogorov $(n,p)$-width]
	Consider a normed vector space $(\linSpace, \norm{\cdot})$ with a subset $\subSet \subseteq \linSpace$.
	Then, the \emph{polynomial Kolmogorov $(n,p)$-width}
    \begin{align*}
    	\polyKolWidthOne[n]{p}{\subSet}{\linSpace}
:= \inf_{\substack{
    \submnf[l]{p} \mathrm{ poly.~mapped~submnf.}\\
    \dim(\submnf[l]{p}) \leq n
}}\,
        \dist(\subSet, \submnf[l]{p})
	\end{align*}
    measures the theoretically optimal worst best-approximation error of $\subSet$ achievable by some polynomially mapped submanifold $\submnf[l]{p}$ of  $\linSpace$ with dimension $\dim(\submnf[l]{p}) \leq l \leq n$.
\end{definition}

This allows us to bound the best-approximation error for ROMs from MOR with polynomially mapped manifolds of order $p$ and reduced dimension $n$ from below with the polynomial Kolmogorov $(n,p)$-width analogously to \eqref{eq:relation_kolmogorov_mor} by
$$\dist(\solmnf, \submnf[n]{p})
\geq \polyKolWidthOne[n]{p}{\solmnf}{\submnf[n]{p}}.$$

We can show, that the polynomial Kolmogorov $(n,p)$-width can be bounded from above and below with quantities relating to the classical Kolmogorov $n$-width. 

\begin{theorem}[Approximation Bounds for Polynomial Kolmogorov $(n,p)$-width]\label{thm:bounds}
	Consider a normed vector space $(\linSpace, \norm{\cdot})$.
	For any set $\subSet \subseteq \linSpace$ 
	the polynomial Kolmogorov $(n,p)$-width for $p \geq 1$ is sandwiched by the classical Kolmogorov $\numTotal{n}{p}$-width and $n$-width, i.e., 
	\begin{align}\label{Eq:bounds}
		\kolWidth[\numTotal{n}{p}]{\subSet}{\linSpace} \leq\polyKolWidthOne[n]{p}{\subSet}{\linSpace} \leq \kolWidth[n]{\subSet}{\linSpace}.
	\end{align}
\end{theorem}

\begin{proof}
	We start by proving the upper bound $\polyKolWidthOne[n]{p}{\subSet}{\linSpace} \leq \kolWidth[n]{\subSet}{\linSpace}$: As $p \geq 1$, linear mappings are included in $\dec[n]{p}$ as a special case by setting all vectors $v_{kj} \in \vectorSet{n}{p}$ with $k \neq 1$ to zero.
	Thus, the polynomially mapped submanifolds $\submnf[l]{p}$ in the polynomial Kolmogorov $(n,p)$-width include all linear subspaces with dimension $l \leq n$ as special case
	and the polynomial Kolmogorov $(n,p)$-width is bounded by the classical Kolmogorov $n$-width from above.

	We continue by deriving the lower bound $\kolWidth[\numTotal{n}{p}]{\subSet}{\linSpace} \leq \polyKolWidthOne[n]{p}{\subSet}{\linSpace}$: From \cref{lem:lin_embed}, we know that $\img(\dec[l]{p}) \subset \subspace{l}{p}$ with $\dim(\subspace{l}{p}) = \numTotal{l}{p}$. Then, for any set $\subSet \subseteq \linSpace$, it holds that
	\begin{align*}
		\dist(\subSet, \submnf[l]{p})
\geq \dist(\subSet, \subspace{l}{p})
\stackrel{\eqref{Eq:kolWidth}}{\geq} \kolWidth[\numTotal{l}{p}]{\subSet}{\linSpace}
\stackrel{l \leq n}{\geq} \kolWidth[\numTotal{n}{p}]{\subSet}{\linSpace}
	\end{align*}
	since $\kolWidth[n]{\subSet}{\linSpace}$ is monotonically decreasing in $n$ and $\numTotal{l}{p}$ is monotonically increasing in $l$. By taking the infimum over $\submnf[l]{p}$ in the left-hand side of the inequality, we arrive at the lower bound.
\end{proof}

\begin{figure}
	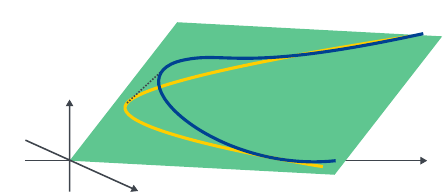
	\centering
	\caption{Schematic illustration of the main assertion of the paper for $n=1$, $p=2$ and ${\nmu}=1$.
	Approximation of the solution manifold $\solmnf$ (solid blue line) with a polynomially mapped submanifold $\submnf[n]{p}$ (solid yellow line) which is contained in the subspace $\subspace{n}{p}$ (green surface). To this end, for a fixed parameter $\mu \in \P$, the approximation of $\bfx(\bfmu)$ by $\submnf[n]{p}$ cannot be better as the approximation of $\bfx(\bfmu)$ by $\subspace{n}{p}$.
	}
		\label{fig:polyKolWidth}
\end{figure}

Especially critical in the statement of \cref{thm:bounds} is the lower bound.
It states that the polynomial mappings are limited by a classical Kolmogorov $n$-width, which results from \cref{lem:lin_embed}: we know that the image of every polynomial of order $p$ is contained in a $ \numTotal{n}{p}$-dimensional subspace of $\linSpace$. This is visualized schematically in \cref{fig:polyKolWidth}, where the solution manifold $\solmnf$ is approximated by the polynomially mapped manifold $\submnf[n]{p}$, which is in turn embedded in the linear subspace $\subspace{n}{p}$. If for a given parameter vector $\bfmu \in \P$, we consider the distance from $\bfx(\bfmu) \in \solmnf$ to $\submnf[n]{p}$, then this particular distance cannot be less than the orthogonal projection of $\bfx(\bfmu) $ onto $\subspace{n}{p}$.

In the following, we show how this impacts certain decay rates in the classical Kolmogorov $n$-widths.

\begin{corollary}\label{corollary:decay_rates}
If the decay of the Kolmogorov $n$-widths is at most algebraic
or exponential
\begin{align}\label{eq:kol_width_decay_rates}
	&\kolWidth[n]{\subSet}{\linSpace} \geq M n^{-\alpha},&
	&\text{or}&
	&\kolWidth[n]{\subSet}{\linSpace} \geq M e^{-a n^\alpha},
\end{align}
for some $M,\alpha,a > 0$, then for $p \ge 2, n \ge 4,$ the decay of the polynomial Kolmogorov $(n,p)$-width is also at most algebraic respective exponential with 
\begin{align*}
	&\polyKolWidthOne[n]{p}{\subSet}{\linSpace} \geq M n^{-\alpha p},&
	&\text{or}&
	&\polyKolWidthOne[n]{p}{\subSet}{\linSpace} \geq M e^{-a n^{\alpha p}}.
\end{align*}
\end{corollary}

\begin{proof}
	We can estimate the total number of vectors used in the polynomial mapping for $p \geq 2$ and $n$ big enough with
	\begin{align}\label{Eq:bound_numTotal}
		\numTotal{n}{p} \leq n^p.
	\end{align}
	For the rigorous proof of this part, we refer due to length to the \cref{appendix}.
	With \cref{thm:bounds} we derive for the
	algebraic case
	\begin{align*}
		\polyKolWidthOne[n]{p}{\subSet}{\linSpace}
	\stackrel{\eqref{Eq:bounds}}{\geq} \kolWidth[\numTotal{n}{p}]{\subSet}{\linSpace}
	\stackrel{\eqref{eq:kol_width_decay_rates}}{\geq} M \left(\numTotal{n}{p} \right)^{-\alpha}
	\stackrel{\eqref{Eq:bound_numTotal}}{\geq} M n^{-\alpha p}
	\end{align*}
	and for the exponential case
	\begin{align*}
		\polyKolWidthOne[n]{p}{\subSet}{\linSpace}
	\stackrel{\eqref{Eq:bounds}}{\geq} \kolWidth[\numTotal{n}{p}]{\subSet}{\linSpace}
	\stackrel{\eqref{eq:kol_width_decay_rates}}{\geq} M e^{-a \left(\numTotal{n}{p} \right)^{\alpha}}
	\stackrel{\eqref{Eq:bound_numTotal}}{\geq}
M e^{-a  n^{\alpha p}}
	\end{align*}
	since both, $(\cdot)^{-\alpha}$ and $e^{-\alpha (\cdot)^{\alpha}}$, are monotonically decreasing.
\end{proof}

For MOR, this theorem means that, if bounds on the classical Kolmogorov $n$-widths $\kolWidth[n]{\solmnf}{\linSpace}$ are known,
then the best-possible approximation error of ROMs based on polynomially mapped manifolds $\polyKolWidthOne[n]{p}{\solmnf}{\linSpace}$ decays within the same type of convergence class.
This is exemplified in the following.

\begin{example}[Linear Advection \cite{ohlberger2016reduced}, Linear Wave Equation \cite{greif2019decay}]
	Two classical results for a provable lower bound of the Kolmogorov $n$-widths are the linear advection model from \cite{ohlberger2016reduced} and the linear wave equation from \cite{greif2019decay}.
	These papers prove that the decay rate of the Kolmogorov $n$-widths of the snapshot set is bounded from below by $1/2 \cdot n^{-1/2}$ for the linear advection problem and, respectively, $1/4 \cdot n^{-1/2}$ for the linear wave equation.
	With \cref{corollary:decay_rates}, we can see that the corresponding polynomial Kolmogorov $(n,p)$-widths of the snapshot set will be limited by
	$$\polyKolWidthOne[n]{p}{\solmnf}{\linSpace} \geq 1/2 \cdot n^{-p/2} \quad \text{and} \quad \polyKolWidthOne[n]{p}{\solmnf}{\linSpace} \geq 1/4 \cdot n^{-p/2},$$
	respectively.
\end{example}

In \cite{Cohen2023}, an alternative formulation of nonlinear widths is introduced, the so-called \emph{manifold widths}. The following definition transfers this concept to our case of polynomially mapped manifolds.

\begin{definition}[Polynomial Manifold $(n,p)$-width]
	Consider a normed vector space $(\linSpace, \norm{\cdot})$ with a subset $\subSet \subset \linSpace$.
	Then, the \emph{polynomial manifold $(n,p)$-width}
	\begin{align*}
		\polyManiWidth[n]{p}{\subSet}{\linSpace}
:= \inf_{\substack{
	\dec[l]{p} \text{ poly.~map.}\\
	l \leq n
}}\;
\inf_{\substack{
	e \in C^0(\linSpace, \field^l)
}}\,
\sup_{s \in \subSet}
		\norm{s - \dec[n]{p}(e(s))}
	\end{align*}
	minimizes the maximum distortion of the encoding procedure of $\subSet$ over the set of all polynomial mappings $\dec[l]{p}$ and some continuous \emph{encoder} $e \in C^0(\linSpace, \field^l)$ with reduced dimension $l \leq n$.
\end{definition}

Indeed, we can show directly from \Cref{thm:bounds} that the lower bound transfers to the polynomial manifold width.

\begin{theorem} \label{th:polynomial_manifold_width}
	The polynomial manifold $(n,p)$-width is bounded from below with
	\begin{align*}
		\kolWidth[\numTotal{n}{p}]{\subSet}{\linSpace}
\leq \polyKolWidthOne[n]{p}{\subSet}{\linSpace}
\leq \polyManiWidth[n]{p}{\subSet}{\linSpace}.
	\end{align*}
\end{theorem}

\begin{proof}
	For a fixed polynomial map $\dec[l]{p}$, it holds for all encoders $e \in C^0(\linSpace, \field^l)$
	\begin{align*}
\sup_{s \in \subSet}
		\norm{s - \dec[l]{p}(
			\underbrace{e(s)}_{\in \field^l}
)}
\geq \inf_{\xred \in \field^l}\,
\sup_{s \in \subSet}
		\norm{s - \dec[l]{p}(\xred)}
	\end{align*}
	and thus
	\begin{align}\label{eq:est_inf_encoder}
		\inf_{\substack{
	e \in C^0(\linSpace, \field^l)
}}\,
\sup_{s \in \subSet}
		\norm{s - \dec[l]{p}(e(s))}
\geq \inf_{\xred \in \field^l}\,
\sup_{s \in \subSet}
		\norm{s - \dec[l]{p}(\xred)}.
	\end{align}
	Together with the max--min inequality (see e.g.\ \cite[Equation (5.46)]{boyd2004convex})
	in step $(\ast)$,
	we observe
	\begin{align*}
		\polyManiWidth[n]{p}{\subSet}{\linSpace}
&= \inf_{\substack{
	\dec[l]{p} \text{ poly.~map.}\\
	l \leq n
}}\;
\inf_{\substack{
	e \in C^0(\linSpace, \field^l)
}}\,
\sup_{s \in \subSet}
		\norm{s - \dec[n]{p}(e(s))}\\
	&\stackrel{\eqref{eq:est_inf_encoder}}{\geq}
\inf_{\substack{
	\dec[l]{p} \text{ poly.~map.}\\
	l \leq n
}}\;
\inf_{\xred \in \field^l}\,
\sup_{s \in \subSet}
\norm{s - \dec[n]{p}(\xred)}\\
	&\stackrel{(\ast)}{\geq}
\inf_{\substack{
	\dec[l]{p} \text{ poly.~map.}\\
	l \leq n
}}\;
\sup_{s \in \subSet}\,
\inf_{\xred \in \field^l}
\norm{s - \dec[n]{p}(\xred)}\\
	&= \polyKolWidthOne[n]{p}{\subSet}{\linSpace}\\
	&\stackrel{\eqref{Eq:bounds}}{\geq} \kolWidth[\numTotal{n}{p}]{\subSet}{\linSpace}
	\end{align*}
	where we use in the second to last step that choosing a polynomially mapped manifold $\submnf[l]{p}$ is equivalent to choosing a polynomial map $\dec[l]{p}$ and choosing an element from $t \in \submnf[l]{p}$ is then equivalent to choosing a reduced coordinate $\xred \in \field^l$.
\end{proof}

Note that the presence of the lower bound allows to transfer the decay rates of Kolmogorov $n$-widths to the polynomial manifold $(n,p)$-widths from \eqref{eq:kol_width_decay_rates} to
\begin{align*}
	&\polyManiWidth[n]{p}{\subSet}{\linSpace} \geq M n^{-\alpha p},&
	&\text{or}&
	&\polyManiWidth[n]{p}{\subSet}{\linSpace} \geq M e^{-a n^{\alpha p}},
\end{align*}
respectively. The proof works analogously to the proof of \Cref{corollary:decay_rates}.

\begin{remark}[Best-Possible Approximation Error vs. Error in Reduced Simulation]
\label{rem:err_in_red_sim}
Note that all the above is a discussion about the best-possible approximation error or respective manifold width.
However, it is by no means guaranteed that the approximation provided by solving the ROM realizes this best-approximation.
Thus, the error between the FOM and the ROM solution, the \emph{error in the reduced simulation}, may be much higher than these two quantities.

This becomes relevant, if one thinks about choosing a ROM based on a polynomially mapping $\dec{p}$ of order $p$ or a (classical) ROM based on an affine linear map $\dec[1]{\numTotal{n}{p}}$.
Although both of these choices share the same lower bound on the best-possible approximation error,
the error in the reduced simulation may differ.
In \cite{BARNETT2022111348}, it is observed experimentally for a complex three-dimensional CFD benchmark problem with a quadratic mapping $\dec[2]{n_2}$ and $n_2=39$ that the error in the reduced simulation behaves as expected from the best-approximation error presented in our paper:
A ROM based on the quadratic mapping $\dec[2]{n_2}$ matches the error in the reduced simulation of a (classical) ROM based on a linear mapping of size $$n_1 = 627 \approx 820 = \numTotal{n_2}{2},$$ 
while at the same time, the offline and online runtimes are improved with the ROM based on the quadratic approach.
This supports the idea that, despite the lower bound on the best-possible approximation error, polynomial mappings are relevant in the application.
\end{remark}

\begin{remark}[Extension to More General Nonlinear Mappings $\mapKronGen{n}$]
Although \cref{thm:bounds} refers to polynomially mapped manifolds only,
it also holds for a more general setting of \emph{nonlinearly mapped reduced coordinates}.
If we consider the structure of $\decGen := \mapLinGen{n} \circ \mapKronGen{n}$ in \cref{lem:lin_embed} as the composition of a general nonlinear map $\mapKronGen{n}: \field^{n} \rightarrow \field^{t}$ for some $t \in \N$ and a linear mapping $\mapLinGen{n}: \field^{t} \rightarrow \linSpace$, one could also choose the linear mapping, e.g., from modes of the proper orthogonal decomposition and the nonlinear mapping (a) as an autoencoder \cite{fresca2022pod} or (b) a general artificial neural network \cite{BARNETT2022b}.
Then, following the ideas of \cref{thm:bounds}, the best-possible approximation error and manifold width over decoders of the proposed form can be bounded from below by the classical Kolmogorov $t$-width $\kolWidth[t]{\solmnf}{\linSpace}$.
\end{remark}

\section{Conclusion}
In this paper, we derived bounds for the best-possible approximation error for ROMs based on polynomially mapped manifolds. If bounds on the Kolmogorov $n$-width of the snapshot set are known, we can derive an upper and a lower bound on the best-possible approximation error depending on the order of the polynomial.
We showed that the class of convergence does not change, but the rate of convergence can be improved, depending on the degree of the polynomial. Future work could be related to the questions on how sharp this lower bound is as well as on what are good algorithms for the construction of polynomially mapped manifolds such that this best-approximation bound could possibly be attained within some tolerance.\\

\acknowledgements{{%
PB and BH are funded by Deutsche Forschungsgemeinschaft (DFG, German Research Foundation) under Germany's Excellence Strategy - EXC 2075 - 390740016 and DFG Project No.~314733389.
We acknowledge the support by the Stuttgart Center for Simulation Science (SimTech).}
}

\printbibliography

\appendix
\section{Appendix} \label{appendix}

We provide a rigorous argument by induction for the first part of the proof of \cref{corollary:decay_rates}, i.e., we show that $\numTotal{n}{p} \le n^{p}$ for all $p\ge 2, n\ge 4$.

\begin{lemma}
For all $p\ge 2, n\ge 4$, it holds that
\begin{align*}
	\left(\numTotal{n}{p} =\ \right) \ \sum_{k=0}^p \binom{n+k-1}{k}  \le n^{p}.
\end{align*}
\end{lemma}

\begin{proof}
We provide a proof by induction over $p$ for all $n \ge 4$. \newline
1) Initial case for $p=2$ and $n \ge 4$ arbitrary: 
\begin{align*}
	\sum_{k=0}^2 \binom{n+k-1}{k} =  \binom{n-1}{0} +  \binom{n}{1} +  \binom{n+1}{2} =  1+ n +\frac{n (n+1)}{2} = 1 + \frac{3}{2} n + \frac{1}{2}n^2.
\end{align*}
We thus need to verify that 
\begin{align*}
	1 + \frac{3}{2} n + \frac{1}{2}n^2 \le n^2,  
\end{align*}
which is equivalent to 
\begin{align*}
	0 \le  \frac{1}{2}n^2 -  \frac{3}{2} n -1 =: \varphi(n). 
\end{align*}
The function $\varphi(n)$ has roots $r_{1,2}= \frac{3 \pm \sqrt{17}}{2}$, with $r_1 < r_2$. As the function $\varphi(n)$ is a parabola with positive curvature, we obtain that $\varphi(r) \ge 0$ for all $r \ge r_2  \approx 3.5616$ and thus $\varphi(n) \ge 0$ for all $n \ge 4$. \newline
2) Induction assumption (IA): We assume that for fixed $p\ge 2$ and for all $ n\ge 4$ it holds that
\begin{align*}
	\sum_{k=0}^p \binom{n+k-1}{k} = \binom{n+p}{p} \le n^{p}.
\end{align*}
3) Induction step $p \rightarrow p+1$, $n \ge 4$ fixed, i.e., we prove that 
\begin{align}\label{Eq:induction_step_p}
	\sum_{k=0}^{p+1} \binom{n+k-1}{k} = \binom{n+p+1}{p+1} \le n^{p+1}.
\end{align}
We start by using the definition of the binomial formula and apply the (IA)
\begin{align*}
	\binom{n+p+1}{p+1} = \frac{n+p+1}{p+1} \binom{n+p}{p} \stackrel{\text{(IA)}}{\le} \frac{n+p+1}{p+1} n^p.
\end{align*}
To finish this part of the proof, we need to show that 
\begin{align*}
\frac{n+p+1}{p+1} \le n,
\end{align*} 
which can be reformulated to 
\begin{align*}
	0 \le n(p+1) - (n+p+1) = p(n-1) -1.
\end{align*}
The latter inequality is fulfilled for all $p\ge2, n\ge 4$, thus \eqref{Eq:induction_step_p} is proven.
\end{proof}

\end{document}

%% file: quadratic_approx_in_linear_space_v2.pdf_tex
\begingroup%
  \makeatletter%
  \providecommand\color[2][]{%
    \errmessage{(Inkscape) Color is used for the text in Inkscape, but the package 'color.sty' is not loaded}%
    \renewcommand\color[2][]{}%
  }%
  \providecommand\transparent[1]{%
    \errmessage{(Inkscape) Transparency is used (non-zero) for the text in Inkscape, but the package 'transparent.sty' is not loaded}%
    \renewcommand\transparent[1]{}%
  }%
  \providecommand\rotatebox[2]{#2}%
  \newcommand*\fsize{\dimexpr\f@size pt\relax}%
  \newcommand*\lineheight[1]{\fontsize{\fsize}{#1\fsize}\selectfont}%
  \ifx\svgwidth\undefined%
    \setlength{\unitlength}{212.2661104bp}%
    \ifx\svgscale\undefined%
      \relax%
    \else%
      \setlength{\unitlength}{\unitlength * \real{\svgscale}}%
    \fi%
  \else%
    \setlength{\unitlength}{\svgwidth}%
  \fi%
  \global\let\svgwidth\undefined%
  \global\let\svgscale\undefined%
  \makeatother%
  \begin{picture}(1,0.43366982)%
    \lineheight{1}%
    \setlength\tabcolsep{0pt}%
    \put(0,0){\includegraphics[width=\unitlength,page=1]{quadratic_approx_in_linear_space_v2.pdf}}%
    \put(0.06158324,0.14975524){\color[rgb]{0.24313725,0.26666667,0.29803922}\makebox(0,0)[lt]{\lineheight{1.25}\smash{\begin{tabular}[t]{l}$\mathbb{R}^N$\end{tabular}}}}%
    \put(0,0){\includegraphics[width=\unitlength,page=2]{quadratic_approx_in_linear_space_v2.pdf}}%
    \put(0.47332985,0.33071926){\color[rgb]{0,0.25490196,0.56862745}\makebox(0,0)[lt]{\lineheight{1.25}\smash{\begin{tabular}[t]{l}$\mathcal{M}$\end{tabular}}}}%
    \put(0.28901373,0.30741541){\color[rgb]{0,0.25490196,0.56862745}\makebox(0,0)[lt]{\lineheight{1.25}\smash{\begin{tabular}[t]{l}$\bm{x}(\bm{\mu})$\end{tabular}}}}%
    \put(0.29535472,0.08071666){\color[rgb]{1,0.80784314,0}\makebox(0,0)[lt]{\lineheight{1.25}\smash{\begin{tabular}[t]{l}$\widetilde{\mathcal{M}}_{n,p}$\end{tabular}}}}%
    \put(0.62846635,0.09183063){\color[rgb]{0,0.25490196,0.56862745}\makebox(0,0)[lt]{\lineheight{1.25}\smash{\begin{tabular}[t]{l}$\bm{x}(\bm{\mu}_1)$\end{tabular}}}}%
    \put(0.6177372,0.3321379){\color[rgb]{0,0.25490196,0.56862745}\makebox(0,0)[lt]{\lineheight{1.25}\smash{\begin{tabular}[t]{l}$\bm{x}(\bm{\mu}_2)$\end{tabular}}}}%
    \put(0.86230915,0.37808272){\color[rgb]{0,0.25490196,0.56862745}\makebox(0,0)[lt]{\lineheight{1.25}\smash{\begin{tabular}[t]{l}$\bm{x}(\bm{\mu}_3)$\end{tabular}}}}%
    \put(0.47304672,0.39439463){\color[rgb]{0.16470588,0.60784314,0.25490196}\makebox(0,0)[lt]{\lineheight{1.25}\smash{\begin{tabular}[t]{l}$\mathcal{A}_{n,p}$\end{tabular}}}}%
    \put(0,0){\includegraphics[width=\unitlength,page=3]{quadratic_approx_in_linear_space_v2.pdf}}%
    \put(0.39876475,0.2){\color[rgb]{0.24313725,0.26666667,0.29803922}\makebox(0,0)[lt]{\lineheight{1.25}\smash{\begin{tabular}[t]{l}$\text{dist}(\bm{x}(\bm{\mu}), \mathcal{A}_{n,p})$\end{tabular}}}}%
    \put(0,0){\includegraphics[width=\unitlength,page=4]{quadratic_approx_in_linear_space_v2.pdf}}%
    \put(-0.07,0.24005041){\color[rgb]{0.24313725,0.26666667,0.29803922}\makebox(0,0)[lt]{\lineheight{1.25}\smash{\begin{tabular}[t]{l}$\text{dist}(\bm{x}(\bm{\mu}),\widetilde{\mathcal{M}}_{n,p})$\end{tabular}}}}%
  \end{picture}%
\endgroup%